\documentclass[a4paper,12pt]{amsart}
\usepackage{amssymb}
\usepackage{amsmath}
\usepackage[all]{xy}

\usepackage[pagebackref]{hyperref}

\newcommand{\R}{{\rm R}}

\newcommand{\ad}{ad}

\newcommand{\Spec}{\text{\rm Spec}\,}
\DeclareMathOperator{\Hom}{Hom}
\DeclareMathOperator{\Aut}{Aut}
\DeclareMathOperator{\Cent}{Cent}

\DeclareMathOperator{\Sym}{Sym}
\DeclareMathOperator{\Ga}{{\mathbf G}_a}
\newcommand{\id}{\text{\rm id}}

\DeclareMathOperator{\ZZ}{{\mathbb Z}}
\DeclareMathOperator{\QQ}{{\mathbb Q}}

\DeclareMathOperator{\Lie}{Lie}

\newtheorem{lem}{Lemma}[section]
\newtheorem*{lem*}{Lemma}
\newtheorem*{thm*}{Theorem}
\newtheorem{thm}[lem]{Theorem}

\theoremstyle{definition}{  \newtheorem{rem}[lem]{Remark}  }
\theoremstyle{definition}{   }
\theoremstyle{definition}{  \newtheorem{defn}[lem]{Definition} }

\usepackage{hyperref}

\newcommand{\st}{\scriptstyle}
\DeclareMathOperator{\Gm}{{\mathbb G}_m}
\newcommand{\RR}{\mathbb R}
\newcommand{\NN}{\mathbb N}

\begin{document}

\title{On Kazhdan's property (T) for isotropic reductive groups}

\author{A.~Stavrova}
\thanks{The author was supported at different stages of her work by the RFBR grants 14-01-31515-mol\_a,
and by the "Native towns" social investment program of PJSC "Gazprom Neft" .}
\address{Chebyshev Laboratory, Department of Mathematics and Mechanics, St. Petersburg State University,
St. Petersburg, Russia}
\email{anastasia.stavrova@gmail.com}

\maketitle

\section{Introduction}
Let $R$ be a connected finitely generated commutative ring with 1, and let $G$ be an isotropic reductive group over $R$ containig a proper
parabolic subgroup $P$. In~\cite{PS} we have constructed a system of generators of the elementary
subgroup $E_P(R)=\left<U_P(R),U_{P^-}(R)\right>$ of $G(R)$ generalizing the standard elementary root unipotents
in Chevalley groups and twisted (or quasi-split) Chevalley groups. These generators $X_\alpha(v)$,
$v\in V_\alpha$, $\alpha\in\Phi_P$, satisfy a generalized Chevalley commutator formula and have other good properties.
Here $\Phi_P$ is a relative root system with respect to $P$, and $V_\alpha$, $\alpha\in\Phi_P$, are certain
finitely generated projective $R$-modules.
In~\cite{EJZK} M. Ershov, A. Jaikin-Zapirain, and M. Kassabov established Kazhdan's property $(T)$
for elementary Chevalley and twisted Chevalley groups, as well as for the corresponding Steinberg groups.
One of the main ingredients of their proof is the result that the set of standard generators is a Kazhdan subset of the corresponding group.
In the present text we extend the latter result to isotropic reductive groups and their standard generators
$X_\alpha(v)$, $v\in V_\alpha$, $\alpha\in\Phi_P$.

\section{Kazhdan's property $(T)$ and groups graded by root systems}

In this section we recall the terminology and results of~\cite{EJZK} related to Kazhdan's property $(T)$
and groups graded by root systems.

Let $G$ be an abstract (discrete) group and $S$ a subset of $G$.

\begin{defn}\cite{EJZK}
\begin{itemize}
\item[(a)] Let $V$ be a unitary representation of $G$.
A nonzero vector $v\in V$ is called \emph{$(S,\epsilon)$-invariant} if
$$\|sv-v\|\leq \epsilon\|v\| \mbox{ for any } s\in S.$$
\item[(b)] Let $V$ be a unitary representation of $G$ without nonzero
invariant vectors. The \emph{Kazhdan constant
$\kappa(G,S,V)$ } is the infimum of the set
$$
\{\epsilon> 0 : V \mbox{ contains an } (S,\epsilon)\mbox{-invariant vector}\}.
$$
\item[(c)] The \emph{Kazhdan constant $\kappa(G,S)$ }of $G$
with
respect to $S$ is the infimum of the set $\{\kappa(G,S,V)\}$ where
$V$ runs over all unitary representations of $G$ without nonzero invariant vectors.
\item[(d)] $S$ is called a \emph{Kazhdan subset of $G$} if
$\kappa(G,S)>0$.
\item[(e)] A group $G$ has \emph{Kazhdan's property $(T)$} if
$G$ has a finite Kazhdan subset.
\end{itemize}
\end{defn}

The main result of our paper establishes that the set of relative root subgroups
of a reductive group scheme is a Kazhdan subset of the elementary subgroup. In order to do that,
we use a criterion stated in~\cite{EJZK} for a very general class of abstract groups that they call
groups graded by root systems.

\begin{defn}\cite{EJZK}\label{def:EJZK-rootsys}
Let $E$ be a real vector space. A finite non-empty subset $\Phi$ of $E$
is called a \emph{root system in $E$}
if
\begin{enumerate}
\item $\Phi$ spans $E$;
\item $\Phi$ does not contain $0$;
\item $\Phi$ is closed under inversion, that is, if $\alpha \in
\Phi$ then $-\alpha \in \Phi$.
\end{enumerate}
The dimension of $E$ is called the \emph{rank of $\Phi$}.
The root system $\Phi$ is called \emph{irreducible} if it cannot be represented
as a disjoint union of two non-empty subsets, whose $\R$-spans have trivial intersection.
If any root of $\Phi$ is contained in an irreducible root subsystem of rank 2, then
$\Phi$ is called~\emph{regular}.
\end{defn}

\begin{rem}
In our work, we customarily use the Bourbaki definition of a root system~\cite{Bou},
which is more restrictive and corresponds to the notion of a \emph{classical root system} in~\cite{EJZK}.
Thus, outside of this section, we will call root systems in the sense of~Definition~\ref{def:EJZK-rootsys}
\emph{root systems in the sense of~\cite{EJZK}}.
\end{rem}

\begin{rem}\label{rem:regsys}
Note that in order for a root system to be regular in the sense of~\cite{EJZK},
it is enough that for any element $\alpha\in\Phi$ there is a non-collinear $\beta\in\Phi$
such that $\Phi$ also contains $x\alpha+y\beta$ for some non-zero $x,y\in\RR$. Indeed,
in this case $\Psi=(\RR\alpha+\RR\beta)\cap\Phi$ is an irreducible root system of rank 2.
\end{rem}

\begin{defn}\cite{EJZK}
Let $\Phi$ be a root system in a space $E$. Let $F=F(\Phi)$ denote the set
of all linear functionals $f: E \to \R$ such that
\begin{enumerate}
\item $f(\alpha) \not=0$ for all $\alpha \in \Phi$;
\item $f(\alpha) \not= f(\beta)$ for any distinct
$\alpha,\beta \in \Phi$.
\end{enumerate}
The sets $\Phi_f=\{\alpha \in \Phi \,|\, f(\alpha)
> 0\}$, $f\in F(\Phi)$, are called the \emph{Borel subsets of $\Phi$}.
\end{defn}

Below we give a definition of the core of a Borel subset which differs slightly from the one in~\cite{EJZK}.
Its equivalence to the original definition follows from~\cite[Lemma 4.4]{EJZK}.

\begin{defn}
The \emph{core} of a Borel subset $\Phi_f$ is the set
$$
C_f=\{ \alpha\in\Phi_f\ |\ \nexists g\in F(\Phi)\colon \RR\alpha=\RR(\Phi_f\cap\Phi_g)\}.
$$
\end{defn}

\begin{defn}\cite{EJZK}
Let $\Phi$ be a root system and $G$ a abstract group. A \emph{$\Phi$-grading of $G$}
is a collection of subgroups
$\{X_\alpha\}_{\alpha\in\Phi}$ of $G$, called \emph{root subgroups}
 such that
\begin{enumerate}
\item $G$ is generated by $\cup X_{\alpha}$;
\item For any $\alpha, \beta\in \Phi$, with $\alpha\not\in\RR_{<0}\beta$,
one has
$$
[X_\alpha,X_\beta] \subseteq \langle X_\gamma \mid \gamma = a \alpha +
b \beta \in\Phi, \ a,b \geq 1 \rangle.
$$
\end{enumerate}
One also says that the grading is \emph{strong},
if for any Borel subset $\Phi_f$, $f\in F(\Phi)$, and any root $\gamma\in C_f$,
one has
$$
X_\gamma \subseteq \langle X_\beta \mid \beta \in \Phi_f\mbox{ and }
\beta\not\in\RR\gamma\rangle.
$$
\end{defn}

\begin{thm}\cite[Theorem 5.1]{EJZK}\label{thm:EJZK-Ksubset}
Let $\Phi$ be a regular root system in the sense of~\cite{EJZK}, and let
$G$ be a group which admits a strong $\Phi$-grading $\{X_{\alpha}\}$.
Then $\cup X_{\alpha}$ is a Kazhdan subset of $G$, and moreover
the Kazhdan constant $\kappa(G,\cup X_\alpha)$ is bounded below
by a constant $\kappa_\Phi$ which depends only on the root system $\Phi$.
\end{thm}

\section{Preliminaries on isotropic reductive groups}

\subsection{Parabolic subgroups and elementary subgroups}

Let $A$ be a commutative ring. Let $G$ be an isotropic reductive group scheme over $A$, and
let $P$ be a parabolic subgroup of $G$ in the sense of~\cite{SGA3}.
Since the base $\Spec A$ is affine, the group $P$ has a Levi subgroup $L_P$~\cite[Exp.~XXVI Cor.~2.3]{SGA3}.
There is a unique parabolic subgroup $P^-$ in $G$ which is opposite to $P$ with respect to $L_P$,
that is $P^-\cap P=L_P$, cf.~\cite[Exp. XXVI Th. 4.3.2]{SGA3}.  We denote by $U_P$ and $U_{P^-}$ the unipotent
radicals of $P$ and $P^-$ respectively.

\begin{defn}
\label{defn:E_P}
The \emph{elementary subgroup $E_P(A)$ corresponding to $P$} is the subgroup of $G(A)$
generated as an abstract group by $U_P(A)$ and $U_{P^-}(A)$.
\end{defn}

Note that if $L'_P$ is another Levi subgroup of $P$,
then $L'_P$ and $L_P$ are conjugate by an element $u\in U_P(A)$~\cite[Exp. XXVI Cor. 1.8]{SGA3}, hence
$E_P(A)$ does not depend on the choice of a Levi subgroup or of an opposite subgroup
$P^-$, respectively. We suppress the particular choice of $L_P$ or $P^-$ in this context.

\begin{defn}
A parabolic subgroup $P$ in $G$ is called
\emph{strictly proper}, if it intersects properly every normal semisimple subgroup of $G$.
\end{defn}

We  will use the following result that is a combination
of~\cite{PS} and~\cite[Exp. XXVI, \S 5]{SGA3}.

\begin{lem}\label{lem:EE}
Let $G$ be a reductive group scheme over a commutative ring $A$, and let $R$ be a commutative $A$-algebra.
Assume that $A$ is a semilocal ring. Then the subgroup $E_P(R)$ of $G(R)$ is the same for any
minimal parabolic $A$-subgroup $P$ of $G$. If, moreover, $G$ contains a strictly proper parabolic $A$-subgroup,
the subgroup $E_P(R)$ is the same for any strictly proper parabolic $A$-subgroup $P$.
\end{lem}
\begin{proof}
See~\cite[Theorem 2.1]{St-poly}.
\end{proof}

\subsection{Torus actions on reductive groups}

Let $R$ be a commutative ring with 1, and let $S=(\Gm_{,R})^N=\Spec(R[x_1^{\pm 1},\ldots,x_N^{\pm 1}])$
be a split $N$-dimensional torus over $R$. Recall that the character group
$X^*(S)=\Hom_R(S,\Gm_{,R})$ of $S$ is canonically isomorphic to $\ZZ^N$.
If $S$ acts $R$-linearly on an $R$-module $V$, this module has a natural $\ZZ^N$-grading
$$
V=\bigoplus_{\lambda\in X^*(S)}V_\lambda,
$$
where
$$
V_\lambda=\{v\in V\ |\ s\cdot v=\lambda(s)v\ \mbox{for any}\ s\in S(R)\}.
$$
Conversely, any $\ZZ^N$-graded $R$-module $V$ can be provided with an $S$-action by the same rule.

Let $G$ be a reductive group scheme over $R$ in the sense of~\cite{SGA3}. Assume that $S$ acts on $G$
by $R$-group automorphisms. The associated Lie algebra functor $\Lie(G)$ then acquires
a $\ZZ^N$-grading compatible with the Lie algebra structure,
$$
\Lie(G)=\bigoplus_{\lambda\in X^*(S)}\Lie(G)_\lambda.
$$

We will use the following version of~\cite[Exp. XXVI Prop. 6.1]{SGA3}.

\begin{lem}\label{lem:T-P}
Let $L=\Cent_G(S)$ be the subscheme of $G$ fixed by $S$. Let
$\Psi\subseteq X^*(S)$ be an $R$-subsheaf of sets closed under addition of characters.

(i) If $0\in\Psi$, then there exists
a unique smooth connected closed subgroup $U_\Psi$ of $G$ containing $L$ and satisfying
\begin{equation}\label{eq:LieUPsi}
\Lie(U_\Psi)=\bigoplus_{\lambda\in\Psi}\Lie(G)_\lambda.
\end{equation}
Moreover, if $\Psi=\{0\}$, then $U_\Psi=L$; if $\Psi=-\Psi$, then $U_\Psi$ is reductive; if $\Psi\cup(-\Psi)=X^*(S)$,
then $U_\Psi$ and $U_{-\Psi}$ are two opposite parabolic subgroups of $G$ with the common Levi subgroup
$U_{\Psi\cap(-\Psi)}$.

(ii) If $0\not\in\Psi$, then there exists a unique smooth connected unipotent closed subgroup $U_\Psi$ of $G$
normalized by $L$ and satisfying~\eqref{eq:LieUPsi}.
\end{lem}
\begin{proof}
The statement immediately follows by faithfully flat descent from the standard facts about the subgroups of
split reductive groups proved in~\cite[Exp. XXII]{SGA3}; see the proof of~\cite[Exp. XXVI Prop. 6.1]{SGA3}.
\end{proof}

\begin{defn}
The sheaf of sets
$$
\Phi=\Phi(S,G)=\{\lambda\in X^*(S)\setminus\{0\}\ |\ \Lie(G)_\lambda\neq 0\}
$$
is called the \emph{system of relative roots of $G$ with respect to $S$}.
\end{defn}

Choosing a total ordering on the $\QQ$-space $\QQ\otimes_{\ZZ} X^*(S)\cong\QQ^n$, one defines the subsets
of positive and negative relative roots $\Phi^+$ and $\Phi^-$, so that $\Phi$ is a disjoint
union of $\Phi^+$, $\Phi^-$, and $\{0\}$. By Lemma~\ref{lem:T-P} the closed subgroups
$$
U_{\Phi^+\cup\{0\}}=P,\qquad U_{\Phi^-\cup\{0\}}=P^-
$$
are two opposite parabolic subgroups of $G$ with the common Levi subgroup $\Cent_G(S)$.
Thus, if a reductive group $G$ over $R$ admits a non-trivial action of a split torus,
then it has a proper parabolic subgroup. The converse is true Zariski-locally, see~Lemma~\ref{lem:relroots} below.

\subsection{Relative roots and subschemes}

In order to prove our main result, we need to use the notions of relative roots and relative root
subschemes. These notions were initially introduced and studied in~\cite{PS}, and further developed
in~\cite{St-serr}.

Let $R$ be a commutative ring. Let $G$ be a reductive group scheme over $R$. Let $P$ be a parabolic subgroup
scheme of $G$ over $R$, and let $L$ be a Levi subgroup of $P$.
By~\cite[Exp. XXII, Prop. 2.8]{SGA3} the root system $\Phi$ of $G_{\overline{k(s)}}$, $s\in\Spec R$,
is constant locally in the Zariski topology on $\Spec R$. The type of the root system of
$L_{\overline{k(s)}}$ is determined by a Dynkin subdiagram
of the Dynkin diagram of $\Phi$, which is also constant Zariski-locally on $\Spec R$
by~\cite[Exp. XXVI, Lemme 1.14 and Prop. 1.15]{SGA3}. In particular, if $\Spec R$ is connected,
all these data are constant on $\Spec R$.

\begin{lem}\label{lem:relroots}\cite[Lemma 3.6]{St-serr}
Let $G$ be a reductive group over a connected commutative ring $R$, $P$ be a parabolic subgroup of $G$, $L$ be a Levi
subgroup of $P$, and $\bar L$ be the image of $L$ under the natural
homomorphism $G\to G^{\ad}\subseteq \Aut(G)$. Let $D$ be the Dynkin diagram of the root system $\Phi$ of
$G_{\overline{k(s)}}$ for any $s\in\Spec A$. We identify $D$ with a set of simple roots of $\Phi$ such that
$P_{\overline{k(s)}}$ is a standard positive parabolic subgroup with respect to $D$. Let $J\subseteq D$
be the set of simple roots such that $D\setminus J\subseteq D$ is the subdiagram corresponing to $L_{\overline{k(s)}}$.
Then there are a unique maximal split subtorus
$S\subseteq\Cent(\bar L)$ and a subgroup $\Gamma\le \Aut(D)$ such that $J$ is invariant under $\Gamma$,
and
for any $s\in\Spec R$ and any split maximal torus $T\subseteq\bar L_{\overline{k(s)}}$
the kernel of the natural surjection
\begin{equation}\label{eq:T-S}
X^*(T)\cong\ZZ\Phi\xrightarrow{\ \pi\ } X^*(S_{\overline{k(s)}})\cong \ZZ\Phi(S,G)
\end{equation}
is generated by all roots $\alpha\in D\setminus J$,
and by all differences $\alpha-\sigma(\alpha)$, $\alpha\in J$, $\sigma\in\Gamma$.
\end{lem}

In~\cite{PS}, we introduced a system of relative roots $\Phi_P$ with respect to a parabolic
subgroup $P$ of a reductive group $G$ over a commutative ring $R$. This system $\Phi_P$ was defined
independently over each member $\Spec A=\Spec A_i$
of a suitable finite disjoint Zariski covering
$$
\Spec R=\coprod\limits_{i=1}^m\Spec A_i,
$$
such that over each $A=A_i$, $1\le i\le m$, the root system $\Phi$ and the Dynkin diagram $D$ of $G$ is constant.
Namely, we considered the formal projection
$$
\pi_{J,\Gamma}\colon\ZZ \Phi
\longrightarrow \ZZ\Phi/\left<D\setminus J;\ \alpha-\sigma(\alpha),\ \alpha\in J,\ \sigma\in\Gamma\right>,
$$
and set $\Phi_P=\Phi_{J,\Gamma}=\pi_{J,\Gamma}(\Phi)\setminus\{0\}$. The last claim of Lemma~\ref{lem:relroots}
allows to identify $\Phi_{J,\Gamma}$ and $\Phi(S,G)$ whenever $\Spec R$ is connected.

\begin{defn}
In the setting of Lemma~\ref{lem:relroots} we call $\Phi(S,G)$ a \emph{system
of relative roots with respect to the parabolic subgroup $P$ over $R$} and denote it by $\Phi_P$.
\end{defn}

If $A$ is a field or a local ring, and $P$ is a minimal parabolic subgroup of $G$,
then $\Phi_P$ is nothing but the relative root system of $G$ with respect to a maximal split subtorus
in the sense of~\cite{BoTi} or, respectively,~\cite[Exp. XXVI \S 7]{SGA3}.

We have also defined in~\cite{PS} irreducible components of systems of relative roots, the subsets of positive and negative
relative roots, simple relative roots, and the height of a root. These definitions are immediate analogs
of the ones for usual abstract root systems, so we do not reproduce them here.

Let $R$ be a commutative ring with 1.
For any finitely generated projective $R$-module $V$, we denote by $W(V)$ the natural affine scheme
over $R$ associated with $V$, see~\cite[Exp. I, \S 4.6]{SGA3}.
Any morphism of $R$-schemes $W(V_1)\to W(V_2)$
is determined by an element $f\in\Sym^*(V_1^\vee)\otimes_R V_2$, where $\Sym^*$ denotes the symmetric algebra,
and $V_1^\vee$ denotes the dual module of $V_1$. If $f\in\Sym^d(V_1^\vee)\otimes_R V_2$,
we say that the corresponding morphism is
homogeneous of degree $d$.
By abuse of notation, we also write $f:V_1\to V_2$ and call it {\it a degree $d$
homogeneous polynomial map from $V_1$ to $V_2$}. In this context, one has
$$
f(\lambda v)=\lambda^d f(v)
$$
for any $v\in V_1$ and $\lambda\in R$.

\begin{lem}\cite[Lemma 3.9]{St-serr}\label{lem:relschemes}.
In the setting of Lemma~\ref{lem:relroots}, for any $\alpha\in\Phi_P=\Phi(S,G)$ there exists a closed
$S$-equivariant  embedding of $R$-schemes
$$
X_\alpha\colon W\bigl(\Lie(G)_\alpha\bigr)\to G,
$$
satisfying the following condition.

\begin{itemize}
\item[\bf{($*$)}] Let $R'/R$ be any ring extension such that $G_{R'}$ is split with
respect to a maximal split $R'$-torus $T\subseteq L_{R'}$. Let $e_\delta$,
$\delta\in\Phi$, be a Chevalley basis of $\Lie(G_{R'})$, adapted to $T$ and $P$, and $x_\delta\colon\Ga\to G_{R'}$, $\delta\in\Phi$, be the associated
system of 1-parameter root subgroups
{\rm(}e.g. $x_\delta=\exp_\delta$ of~\cite[Exp. XXII, Th. 1.1]{SGA3}{\rm)}.
Let
$$
\pi:\Phi=\Phi(T,G_{R'})\to\Phi_P\cup\{0\}
$$
be the natural projection.
Then for any
$u=\hspace{-8pt}\sum\limits_{\delta\in\pi^{-1}(\alpha)}\hspace{-8pt}a_\delta e_\delta\in\Lie(G_{R'})_\alpha$
one has
\begin{equation}\label{eq:Xalpha-prod}
X_\alpha(u)=
\Bigl(\prod\limits_{\delta\in\pi^{-1}(\alpha)}\hspace{-8pt}x_{\delta}(a_\delta)\Bigr)\cdot
\prod\limits_{i\ge 2}\Bigl(\prod\limits_{
\st
\theta\in \pi^{-1}(i\alpha)
}
\hspace{-8pt}x_\theta(p^i_{\theta}(u))\Bigr),
\end{equation}
where every $p^i_{\theta}:\Lie(G_{R'})_\alpha\to R'$ is a homogeneous polynomial map of degree $i$,
and the products over $\delta$ and $\theta$ are taken in any fixed order.
\end{itemize}
\end{lem}

\begin{defn}
Closed embeddings $X_\alpha$, $\alpha\in\Phi_P$, satisfying the statement of Lemma~\ref{lem:relschemes},
are called \emph{relative root subschemes of $G$ with respect to the parabolic subgroup $P$}.
\end{defn}

Relative root subschemes of $G$ with respect to $P$, actually,
depend on the choice of a Levi subgroup $L$ in $P$, but their essential properties stay the same,
so we usually omit $L$ from the notation.

We will use the following properties of relative root subschemes.

\begin{lem}\label{lem:rootels}\cite[Theorem 2, Lemma 6, Lemma 9]{PS}
Let $X_\alpha$, $\alpha\in\Phi_P$, be as in Lemma~\ref{lem:relschemes}.
Set $V_\alpha=\Lie(G)_\alpha$ for short. Then

(i) There exist degree $i$ homogeneous polynomial maps $q^i_\alpha:V_\alpha\oplus V_\alpha\to V_{i\alpha}$, i>1,
such that for any $R$-algebra $R'$ and for any
$v,w\in V_\alpha\otimes_R R'$ one has
\begin{equation}\label{eq:sum}
X_\alpha(v)X_\alpha(w)=X_\alpha(v+w)\prod_{i>1}X_{i\alpha}\left(q^i_\alpha(v,w)\right).
\end{equation}

(ii) For any $g\in L(R)$, there exist degree $i$ homogeneous polynomial maps
$\varphi^i_{g,\alpha}\colon V_\alpha\to V_{i\alpha}$, $i\ge 1$, such that for any $R$-algebra $R'$ and for any
$v\in V_\alpha\otimes_R R'$ one has
$$
gX_\alpha(v)g^{-1}=\prod_{i\ge 1}X_{i\alpha}\left(\varphi^i_{g,\alpha}(v)\right).
$$

(iii) \emph{(generalized Chevalley commutator formula)} For any $\alpha,\beta\in\Phi_P$
such that $m\alpha\neq -k\beta$ for all $m,k\ge 1$,
there exist polynomial maps
$$
N_{\alpha\beta ij}\colon V_\alpha\times V_\beta\to V_{i\alpha+j\beta},\ i,j\ge 1,
$$
homogeneous of degree $i$ in the first variable and of degree $j$ in the second
variable, such that for any $R$-algebra $R'$ and for any
for any $u\in V_\alpha\otimes_R R'$, $v\in V_\beta\otimes_R R'$ one has
\begin{equation}\label{eq:Chev}
[X_\alpha(u),X_\beta(v)]=\prod_{i,j\ge 1}X_{i\alpha+j\beta}\bigl(N_{\alpha\beta ij}(u,v)\bigr)
\end{equation}

(iv) For any subset $\Psi\subseteq X^*(S)\setminus\{0\}$ that is closed under addition,
the morphism
$$
X_\Psi\colon W\Bigl(\,\bigoplus_{\alpha\in\Psi}V_\alpha\Bigr)\to U_\Psi,\qquad
(v_\alpha)_\alpha\mapsto\prod_\alpha X_\alpha(v_\alpha),
$$
where the product is taken in any fixed order,
is an isomorphism of schemes.
\end{lem}

Lemma~\ref{lem:rootels} immediately implies that for any $R$-algebra $R'$, one has
$$
U_{P^\pm}(R')=\left<X_\alpha(R'\otimes_R V_\alpha),\ \alpha\in\Phi_P^\pm\right>,
$$
and
$$
E_P(R')=\left<X_\alpha(R'\otimes_R V_\alpha),\ \alpha\in\Phi_P\right>.
$$

For any $\alpha\in\Phi_P$, consider the set of relative roots
$$
(\alpha)=\{i\alpha\ |\ i\in\ZZ,\ i\ge 1\}\subseteq\ZZ\Phi_P.
$$
Then the above equality can be rewritten as
\begin{equation}
E_P(R)=\left<X_{(\alpha)},\ \alpha\in\Phi_P\right>.
\end{equation}

\section{Main Theorem}

Our main result is the following theorem.

\begin{thm}\label{thm:Ksubset}
Let $G$ be a reductive group scheme over a connected finitely generated commutative ring $R$. Let $P$ be a proper
parabolic $R$-subgroup scheme of $G$ such that the rank of every irreducible component of $\Phi_P$ is $\ge 2$. The set
$\bigcup\limits_{\alpha\in\Phi_P}X_{(\alpha)}$ is a Kazhdan subset of $E_P(R)$.
\end{thm}

This result is deduced by applying Theorem~\ref{thm:EJZK-Ksubset} of
M. Ershov, A. Jaikin-Zapirain, and M. Kassabov. We need to check that all conditions of this theorem are satisfied.

\begin{lem}\label{lem:Cfrel}
Let $\Phi$ be a root system in the sense of~\cite{Bou}. Let $D$ be a set of simple roots in $\Phi$;
we identify $D$ with the corresponding Dynkin diagram.
Take $J\subseteq D$.
Consider the formal projection
$$
\pi_J\colon\ZZ \Phi
\longrightarrow \ZZ\Phi/\ZZ(D\setminus J),
$$
and set $\Phi_J=\pi_J(\Phi)\setminus\{0\}$. Then

(i) The set $\Phi_J$ is a regular root system in the sense of~\cite{EJZK}
if and only if the rank of every irreducible component of  $\Phi_J$ is $\ge 2$.

(ii) For any $f\in F(\Phi_J)$ there is a set of simple roots $D'$ of $\Phi$ and a subset $J'\subseteq D'$
such that $\ZZ(D'\setminus J')\cap\Phi=\ZZ(D\setminus J) \cap\Phi$, and
$$
(\Phi_J)_f=\pi_{J}\bigl(\Phi^{+'}\bigr)\setminus\{0\}=\pi_{J'}\bigl(\Phi^{+'}\bigr)\setminus\{0\},
$$
where $\Phi^{+'}$ denotes the set of positive roots with respect to $D'$.

(iii) In the above setting, one has
$$
(\Phi_J)_f\setminus C_f=\Phi_J\cap\bigcup_{\alpha\in\pi_{J}(J')} \NN\alpha.
$$
\end{lem}
\begin{proof}
(i) It is clear that if at least one component of $\Phi_J$ has rank 1, then $\Phi_J$ is not regular.
The opposite follows from Remark~\ref{rem:regsys} and~\cite[Lemma 5]{PS}.

(ii) Extend $\pi_{J,\Gamma}$ to $\RR\Phi$ by linearity, and set $g=f\circ\pi_J:\RR\Phi\to\RR$.
Then the sets $\{\alpha\in\Phi\ |\ g(\alpha)\ge 0\}$ and $\{\alpha\in\Phi\ |\ g(\alpha)\le 0\}$ are two
opposite parabolic sets of roots in $\Phi$ in the sense of~\cite{Bou}. Therefore, there is a set of simple
roots $D'$ in $\Phi$ and a subset $J'\subseteq D'$ such that
$$
\begin{array}{c}
\{\alpha\in\Phi\ |\ g(\alpha)\ge 0\}=\Phi^{+'}\cup\bigl(\Phi\cap\ZZ(D'\setminus J')\bigr);\\
\{\alpha\in\Phi\ |\ g(\alpha)\le 0\}=\Phi^{-'}\cup\bigl(\Phi\cap\ZZ(D'\setminus J')\bigr).\\
\end{array}
$$
Clearly, $D',J'$ are as required.

(iii) We can assume that $D=D'$ and $J=J'$ without loss of generality. Then $(\Phi_J)_f=\pi_J(\Phi^+)\setminus\{0\}$.
By~\cite[Lemma 4.6]{EJZK} any $\alpha\in(\Phi_J)_f$ that is a positive linear combination of $\ge 2$ linearly independent
elements of $(\Phi_J)_f$, belongs to $C_f$. Then
$(\Phi_J)_f\setminus C_f\subseteq\bigcup_{\alpha\in\pi_{J}(J)} \NN\alpha$. It remains to prove
that every positive multiple of every $\alpha\in\pi_J(J)$ belongs to $(\Phi_J)_f\setminus C_f$.
Since all elements of $\pi(J)$ are linearly independent
and integrally span $\Phi_J$ by definition of $\pi_J$
and $D$, we can define a linear map $g:\RR\Phi_J\to\RR$ by setting $g(\alpha)=1$, $g(\beta)=-|\Phi_J|-1$ for all
$\beta\in \pi(J)\setminus\{\alpha\}$. Clearly,
for any $\gamma\in(\Phi_J)_f=\pi_J(\Phi^+)$, if $\gamma\in\NN\alpha$, then $g(\gamma)>0$,
and otherwise $g(\gamma)<1\cdot |\Phi_J|-|\Phi_J|-1=-1$. In particular, $g(\gamma)\neq 0$ for all $\gamma\in\Phi_J$,
and
$\RR((\Phi_J)_g\cap(\Phi_J)_f)=\RR\alpha$.
\end{proof}

\begin{lem}\label{lem:nc-roots}
In the setting of Lemma~\ref{lem:Cfrel}, let $\Phi^+$ denote the set of positive roots with respect to $D$.
Let $\alpha\in\pi_J(\Phi^+)\setminus\{0\}$ be such that $\alpha\neq m\alpha_0$ for any $m\ge 1$ and
$\alpha_0\in\pi_J(D)\setminus\{0\}$. Then there are non-collinear
$\beta\in\pi_J(\Phi^+)\setminus\{0\}$ and $\gamma\in\pi_J(J)$ such that $\alpha=\beta+\gamma$.
\end{lem}
\begin{proof}
For any $x\in\pi_J^{-1}(\alpha)$
there exists a sequence of simple roots $y_1,\ldots,y_n\in D$ such that
$x=y_1+\ldots+y_n$ and $y_1+\ldots+y_i\in\Phi$ for any $1\le i\le n$.
Let $i$ be the least possible index such that $y_{i+1},\ldots,y_n\in D\setminus J$. Then $y_i\in J$
and $\pi_J(y_1+\ldots+y_{i-1}+y_i)=\alpha$. Set $\beta=\pi_J(y_1+\ldots+y_{i-1})$ and
$\gamma=\pi_J(y_i)$. The relative roots $\beta$ and $\gamma$ are non-collinear since
otherwise we would have had $\alpha=k\pi_J(y_i)$ for some $k\ge 1$.
\end{proof}

\begin{lem}\label{lem:ABC}
Let $G$ be a reductive group scheme over a connected finitely generated commutative ring $R$. Let $P$ be a proper
parabolic $R$-subgroup scheme of $G$.
Suppose that $\alpha,\beta,\gamma\in\Phi_P$ are pairwise non-collinear roots such that
$\alpha=\beta+\gamma$. If the absolute root system $\Phi$ of $G$ has a component of type $G_2$, assume
moreover that $\beta-\gamma\not\in\Phi_P$. Then
\begin{multline*}
X_\alpha(V_\alpha)\subseteq
\langle \{X_{i\beta+j\gamma}(V_{i\beta+j\gamma})\ |\ i,j\ge 0,\ (i,j)\neq (1,1),\ i\beta+j\gamma\in\Phi_P\}\\
\cup
\{X_{i(\beta-\gamma)+j\gamma}(V_{i(\beta-\gamma)+j\gamma})\ |\ i,j\ge 0,\ (i,j)\neq (1,2),\ i(\beta-\gamma)+j\gamma\in\Phi_P\}
\rangle.
\end{multline*}
\end{lem}
\begin{proof}
The statement is proved exactly as Lemma~\cite[Lemma 11]{PS}, using~\cite[Lemma 10]{PS} and~\eqref{eq:Chev}.
\end{proof}

\begin{proof}[Proof of Theorem~\ref{thm:Ksubset}.]
By the generalized Chevalley commutator formula~\eqref{eq:Chev} of Lemma~\ref{lem:rootels}, the groups
$X_{(\alpha)}$, $\alpha\in\Phi_P$, constitute a $\Phi_P$-grading of $E_P(R)$.
By Theorem~\ref{thm:EJZK-Ksubset} of M. Ershov, A. Jaikin-Zapirain, and M. Kassabov, it remains to check that
this grading is strong.

In Lemma~\ref{lem:relroots} we identified $\Phi_P=\Phi(S,G)$ with a relative root system
$\Phi_{J,\Gamma}=\pi_{J,\Gamma}(\Phi)\setminus\{0\}$,
where
$$
\pi_{J,\Gamma}\colon\ZZ \Phi
\longrightarrow \ZZ\Phi/\left<D\setminus J;\ \alpha-\sigma(\alpha),\ \alpha\in J,\ \sigma\in\Gamma\right>,
$$
is the natural projection. Observe that for any root system $\Phi$ in the sense of~\cite{Bou} and
any $\Gamma\le\Aut(D)$ as above, the relative root system
$\Psi=\Phi_{D,\Gamma}=\pi_{D,\Gamma}(\Phi)\setminus\{0\}$ is isomorphic (as a set with partially defined addition)
to another root system in the sense of~\cite{Bou}, with a basis of simple roots $\pi_{D,\Gamma}(D)$.
Namely, If (an irreducible component of) $\Phi$
has type $A_l$, then $\Phi_{D,\Gamma}$ is a root system of type $BC_{l/2}$ or $B_{l+1/2}$; if $D$ is of type
$E_6$, then $\Phi_{D,\Gamma}$ is of type $F_4$, etc. Furthermore, the composition
\begin{multline*}
\pi_{\pi_{D,\Gamma}(J),\id}\circ\pi_{D,\Gamma}\colon \ZZ\Phi
\longrightarrow \ZZ\Phi/\left<\alpha-\sigma(\alpha),\ \alpha\in D,\ \sigma\in\Gamma\right>=\ZZ\Psi\\
\longrightarrow \ZZ\Phi/\left<D\setminus J;\ \alpha-\sigma(\alpha),\ \alpha\in J,\ \sigma\in\Gamma\right>
\end{multline*}
coincides with $\pi_{J,\Gamma}$. Therefore, $\Phi_P=\Psi_{\pi_{D,\Gamma}(J)}$ is subject to Lemma~\ref{lem:Cfrel}.
In particular, by Lemma~\ref{lem:Cfrel} (i) $\Phi_P$ is a regular root system in the sense of~\cite{EJZK}.
Furthermore, by Lemma~\ref{lem:Cfrel} (ii) and (iii) for any Borel subset $(\Phi_P)_f$ of $\Phi_P$ there is a set of simple roots
$E$ of $\Psi$ and a subset $I\subseteq E$ such that $(\Phi_P)_f=\pi_I(\Psi^+)$, where $\Psi^+$
is the set of simple roots of $\Psi$ with respect to $E$, and
$$
(\Phi_P)_f\setminus C_f=\Phi_P\cap\bigcup_{\alpha\in\pi_{I}(I)} \NN\alpha.
$$

Now pick any $\gamma\in C_f$.
We need to show that
\begin{equation}\label{eq:xg}
X_{(\gamma)} \subseteq \langle X_\beta \mid \beta \in (\Phi_P)_f\mbox{ and }
\beta\not\in\RR\gamma\rangle.
\end{equation}
By the above, $\gamma$ is not a positive multiple of any $\alpha_0\in\pi_I(I)$.
By Lemma~\ref{lem:nc-roots} for any $k\ge 1$ there are such non-collinear
$\beta_k,\gamma_k\in\pi_I(\Psi^+)\setminus\{0\}=(\Phi_P)_f$
that $k\gamma=\beta_k+\gamma_k$. Observe that if $\pi_I^{-1}(\gamma)$ intersects an irreducible
component of $\Psi$ of type $G_2$, then, since $\gamma$ lies in a component of $\Phi_P$ of rank $\ge 2$,
we can assume without loss of generality that $G$ is a quasi-split group of type $G_2$ or ${}^{3(6)}D_4$,
and then the statement follows from the results of~\cite{EJZK} on Chevalley groups. So we can assume that
$\pi_I^{-1}(\gamma)$ does not intersect any irreducible component of $\Psi$ of type $G_2$. Then by
Lemma~\ref{lem:ABC} $X_{k\gamma}(V_{k\gamma})$ belongs to the subgroup generated by
$X_{(i\beta_k+j\gamma_k)}$, $i,j\ge 0$, $(i,j)\neq (1,1)$, $i\beta_k+j\gamma_k\in\Phi_P$,
 and $X_{(i(\beta_k-\gamma_k)+j\gamma_k)}$,
$i,j\ge 0$, $(i,j)\neq (1,2)$, $i(\beta_k-\gamma_k)+j\gamma_k\in\Phi_P$.
Clearly, $i\beta_k+j\gamma_k\in(\Phi_P)_f$, if $i,j\ge 0$. On the other hand,
since, clearly, $\beta_k$ includes a simple relative root different from $\gamma$,
if $i(\beta_k-\gamma_k)+j\gamma_k\in\Phi_P$, then $j-i\ge 0$, and hence also $i(\beta_k-\gamma_k)+j\gamma_k\in(\Phi_P)_f$.
Running a descending induction over the (finite) set of $k\ge 1$ such that $k\gamma\in\Phi_P$, we conclude~\eqref{eq:xg}. The theorem is proved.
\end{proof}

\renewcommand{\refname}{References}

\end{document}